\documentclass[reqno,b5paper]{amsart}
\usepackage{amsmath,amsthm,amsfonts,amssymb,mathrsfs}
\usepackage{amsmath}
\usepackage{amssymb}
\usepackage{amsthm}
\usepackage{enumerate}
\usepackage[mathscr]{eucal}
\usepackage{eqlist}
\usepackage{color}

\input{xy}
 \xyoption{all}
 \xyoption{arc}

\theoremstyle{plain}
\newtheorem{theorem}{Theorem}[section]
\newtheorem{proposition}[theorem]{Proposition}
\newtheorem{corollary}[theorem]{Corollary}

\theoremstyle{definition}

\newtheorem{remark}{\textnormal{\textbf{Remark}}}
\newtheorem*{acknowledgement}{\textnormal{\textbf{Acknowledgement}}}

\theoremstyle{remark}

\numberwithin{equation}{section}

\usepackage{hyperref}

\begin{document}

\title[On monoids of injective partial
cofinite selfmaps]{On monoids of injective partial cofinite
selfmaps}

\author[Oleg~Gutik \and Du\v{s}an~Repov\v{s}]{Oleg~Gutik* \and Du\v{s}an~Repov\v{s}**}

\newcommand{\acr}{\newline\indent}

\address{\llap{*\,}Department of Mechanics and Mathematics\acr
                   Ivan Franko National University of Lviv\acr
                   Universytetska 1\acr
                   Lviv, 79000\acr UKRAINE}
\email{{o\_gutik}@franko.lviv.ua, {ovgutik}@yahoo.com}

\address{\llap{**\,}Faculty of Education, and \acr
                    Faculty of Mathematics and Physics\acr
                    University of Ljubljana\acr
                    Kardeljeva plo\v{s}\v{c}ad 16\acr
                    Ljubljana, 1000\acr
                    SLOVENIA}
\email{dusan.repovs@guest.arnes.si}

\thanks{This research was supported by the Slovenian Research Agency grants
P1-0292-0101,  J1-5435-0101 and J1-6721-0101.}

\keywords{Bicyclic semigroup, semigroup of
bijective partial transformations, congruence, symmetric group, group congruence, semidirect product.
 }
\subjclass[2010]{Primary 20M20, 20M18;
Secondary 20B30}

\begin{abstract}
We study the semigroup $\mathscr{I}^{\mathrm{cf}}_\lambda$ of
injective partial cofinite selfmaps of an infinite cardinal $\lambda$.
We show that $\mathscr{I}^{\mathrm{cf}}_\lambda$ is a bisimple
inverse semigroup and each chain of idempotents in $\mathscr{I}^{\mathrm{cf}}_\lambda$ is contained in a bicyclic subsemigroup of $\mathscr{I}^{\mathrm{cf}}_\lambda$,
we describe the Green relations on
$\mathscr{I}^{\mathrm{cf}}_\lambda$ and we prove that every
non-trivial congruence on $\mathscr{I}^{\mathrm{cf}}_\lambda$ is a
group congruence. Also, we describe the structure of the quotient semigroup $\mathscr{I}^{\mathrm{cf}}_\lambda/\sigma$, where $\sigma$ is the least group congruence on $\mathscr{I}^{\mathrm{cf}}_\lambda$.
\end{abstract}

\maketitle

\section{Introduction and preliminaries}

In this paper we shall denote the first infinite ordinal by $\omega$ and the cardinality of the set $A$ by $|A|$. We shall identify all cardinals with
their corresponding initial ordinals.  We shall denote the set of integers by $\mathbb{Z}$ and the additive group of integers by $\mathbb{Z}(+)$.

A semigroup $S$ is called {\it inverse} if for every element $x\in S$ there exists a unique $x^{-1}\in S$ such that $xx^{-1}x=x$ and $x^{-1}xx^{-1}=x^{-1}$. The element $x^{-1}$ is called the {\it inverse of} $x\in S$. If $S$ is an inverse semigroup, then the function $\operatorname{inv}\colon S\to S$ which assigns to every element $x$ of $S$ its inverse element $x^{-1}$ is called the {\it inversion}.

A congruence $\mathfrak{C}$ on a semigroup $S$ is called \emph{non-trivial} if $\mathfrak{C}$ is distinct from the universal and the identity congruences on $S$, and a \emph{group congruence} if the quotient semigroup $S/\mathfrak{C}$ is a group.

If $S$ is a semigroup, then we shall denote the subset of all idempotents in $S$ by $E(S)$. If $S$ is an inverse semigroup, then $E(S)$ is closed under multiplication and we shall refer to $E(S)$ as a \emph{band} (or the \emph{band of} $S$). Then the semigroup operation on $S$ determines the following partial order $\leqslant$
on $E(S)$: $e\leqslant f$ if and only if $ef=fe=e$. This order is called the {\em natural partial order} on $E(S)$. A \emph{semilattice} is a commutative semigroup of idempotents. A semilattice $E$ is called {\em linearly ordered} or a \emph{chain}
if its natural  order is a linear order. A \emph{maximal chain} of a semilattice $E$ is a chain which is not properly contained in any other chain of $E$.

The Axiom of Choice implies the existence of maximal chains in every
partially ordered set. According to
\cite[Definition~II.5.12]{Petrich1984}, a chain $L$ is called an
\emph{$\omega$-chain} if $L$ is isomorphic to $\{0,-1,-2,-3,\ldots\}$ with
the usual order $\leqslant$ or equivalently, if $L$ is isomorphic to
$(\omega,\max)$. Let $E$ be a semilattice and $e\in E$. We put
${\downarrow} e=\{ f\in E\mid f\leqslant e\}$ and ${\uparrow} e=\{
f\in E\mid e\leqslant f\}$.  By
$(\mathscr{P}_{<\omega}(\lambda),\cup)$ we shall denote the
\emph{free semilattice with identity} over a set of cardinality
$\lambda\geqslant\omega$, i.e.,
$(\mathscr{P}_{<\omega}(\lambda),\cup)$ is the set of all finite
subsets (with the empty set) of $\lambda$ with the semilattice
operation ``union''.

If $S$ is a semigroup, then we shall denote the Green relations on $S$ by $\mathscr{R}$, $\mathscr{L}$, $\mathscr{J}$, $\mathscr{D}$ and $\mathscr{H}$ (see \cite{CP}). A semigroup $S$ is called \emph{simple} if $S$ does not contain proper two-sided ideals and \emph{bisimple} if $S$ has only one $\mathscr{D}$-class.

The bicyclic semigroup ${\mathscr{C}}(p,q)$ is the semigroup with the identity $1$ generated by elements $p$ and $q$ subject only to the condition $pq=1$. The bicyclic semigroup is bisimple and every one of its congruences is either trivial or a group congruence. Moreover, every homomorphism $h$ of the bicyclic semigroup is either an isomorphism or the image of ${\mathscr{C}}(p,q)$ under $h$ is a cyclic group~(see \cite[Corollary~1.32]{CP}). The bicyclic semigroup plays an important role in algebraic theory of semigroups and in the theory of topological semigroups. For example a well-known Andersen's result~\cite{Andersen} states that a ($0$--)simple semigroup is completely ($0$--)simple if and only if it does not contain the bicyclic semigroup. The bicyclic semigroup admits only the discrete topology~\cite{EberhartSelden1969}. The problem of embeddability of the bicycle semigroup into compact-like semigroups was studied in \cite{AHK, BanakhDimitrovaGutik2009, BanakhDimitrovaGutik2010, GutikRepovs2007,
HildebrantKoch1988}.

\begin{remark}\label{remark-1.1}
We observe that the bicyclic semigroup is isomorphic to the
semigroup $\mathscr{C}_{\mathbb{N}}(\alpha,\beta)$ which is
generated by partial transformations $\alpha$ and $\beta$ of the set
of positive integers $\mathbb{N}$, defined as follows:
$(n)\alpha=n+1$ if $n\geqslant 1$ and $(n)\beta=n-1$ if $n> 1$ (see Exercise~IV.1.11$(ii)$ in \cite{Petrich1984}).
\end{remark}

If $T$ is a semigroup, then we say that a subsemigroup $S$ of $T$ is a \emph{bicyclic subsemigroup of} $T$ if $S$ is isomorphic to the bicyclic semigroup ${\mathscr{C}}(p,q)$.

Hereafter we shall assume that $\lambda$ is an infinite cardinal.
If $\alpha\colon X\rightharpoonup Y$ is a partial map, then we shall denote
the domain and the range of $\alpha$ by $\operatorname{dom}\alpha$ and $\operatorname{ran}\alpha$, respectively.

Let $\mathscr{I}_\lambda$ denote the set of all partial one-to-one
transformations of an infinite set $X$ of cardinality $\lambda$
together with the following semigroup operation:
\begin{equation*}
x(\alpha\beta)=(x\alpha)\beta \quad \hbox{if} \quad
x\in\operatorname{dom}(\alpha\beta)=\left\{
y\in\operatorname{dom}\alpha\mid
y\alpha\in\operatorname{dom}\beta\right\},\quad \hbox{for }
\alpha,\beta\in\mathscr{I}_\lambda.
\end{equation*}
The semigroup $\mathscr{I}_\lambda$ is called the \emph{symmetric
inverse semigroup} over the set $X$~(see \cite[Section~1.9]{CP}).
The symmetric inverse semigroup was introduced by
Vagner~\cite{Wagner1952} and it plays a major role in the theory of
semigroups.

Furthermore, we shall identify the cardinal $\lambda=|X|$ with the
set $X$. By $\mathscr{I}^{\mathrm{cf}}_\lambda$ we shall denote a
subsemigroup of injective partial selfmaps of $\lambda$ with
cofinite domains and ranges in $\mathscr{I}_\lambda$, i.e.,
\begin{equation*}
 \mathscr{I}^{\mathrm{cf}}_\lambda=\left\{\alpha\in\mathscr{I}_\lambda\mid
 |\lambda\setminus\operatorname{dom}\alpha|<\infty \quad \mbox{and} \quad
 |\lambda\setminus\operatorname{ran}\alpha|<\infty\right\}.
\end{equation*}
Obviously, $\mathscr{I}^{\mathrm{cf}}_\lambda$ is an inverse
submonoid of the semigroup $\mathscr{I}_\lambda$. We shall call the
semigroup $\mathscr{I}^{\mathrm{cf}}_\lambda$  the \emph{monoid of
injective partial cofinite selfmaps} of $\lambda$.

Next, by $\mathbb{I}$ we shall denote the identity and by
$H(\mathbb{I})$ the group of units of the semigroup
$\mathscr{I}^{\mathrm{cf}}_\lambda$.

It well known that each partial injective cofinite selfmap $f$ of $\lambda$ induces a homeomorphism $f^*\colon\lambda^*\rightarrow\lambda^*$ of the remainder $\lambda^*=\beta\lambda\setminus\lambda$ of the Stone-\v{C}ech compactification of the discrete space $\lambda$. Moreover, under some set theoretic axioms (like \textbf{PFA} or \textbf{OCA}), each homeomorphism of $\omega^*$ is induced by some partial injective cofinite selfmap of $\omega$ (see  \cite{ShelahSteprans1989}--\cite{Velickovic1993}). So the inverse semigroup  $\mathscr{I}^{\mathrm{cf}}_\lambda$  admits a natural homomorphism  $\mathfrak{h}\colon \mathscr{I}^{\mathrm{cf}}_\lambda\rightarrow \mathscr{H}(\lambda^*)$ to the homeomorphism group $\mathscr{H}(\lambda^*)$ of $\lambda^*$ and this homomorphism is surjective under certain set theoretic assumptions.

The semigroups $\mathscr{I}_{\infty}^{\!\nearrow}(\mathbb{N})$ and $\mathscr{I}_{\infty}^{\!\nearrow}(\mathbb{Z})$ of injective isotone partial selfmaps with cofinite domains and images of positive integers and integers, respectively, were studied in \cite{GutikRepovs2011} and \cite{GutikRepovs2012x}. There it was proved that the semigroups $\mathscr{I}_{\infty}^{\!\nearrow}(\mathbb{N})$ and $\mathscr{I}_{\infty}^{\!\nearrow}(\mathbb{Z})$ have properties similar to the bicyclic semigroup: they are bisimple and  every non-trivial homomorphic image of $\mathscr{I}_{\infty}^{\!\nearrow}(\mathbb{N})$ and $\mathscr{I}_{\infty}^{\!\nearrow}(\mathbb{Z})$ is a group, and moreover, the semigroup $\mathscr{I}_{\infty}^{\!\nearrow}(\mathbb{N})$ has $\mathbb{Z}(+)$ as a maximal group image and $\mathscr{I}_{\infty}^{\!\nearrow}(\mathbb{Z})$ has $\mathbb{Z}(+)\times\mathbb{Z}(+)$, respectively.

In this paper we shall study algebraic properties of the semigroup
$\mathscr{I}^{\mathrm{cf}}_\lambda$. We shall show that
$\mathscr{I}^{\mathrm{cf}}_\lambda$ is a bisimple inverse semigroup
and every chain of idempotents in $\mathscr{I}^{\mathrm{cf}}_\lambda$ is contained in a bicyclic subsemigroup of $\mathscr{I}^{\mathrm{cf}}_\lambda$, we shall describe the Green relations on $\mathscr{I}^{\mathrm{cf}}_\lambda$ and we shall prove that every non-trivial congruence on $\mathscr{I}^{\mathrm{cf}}_\lambda$ is a group congruence. Also, we shall describe the structure of the quotient semigroup $\mathscr{I}^{\mathrm{cf}}_\lambda/\sigma$, where $\sigma$ is the least group congruence on $\mathscr{I}^{\mathrm{cf}}_\lambda$.


\section{Algebraic properties of the semigroup
$\mathscr{I}^{\mathrm{cf}}_\lambda$}

\begin{proposition}\label{proposition-2.1} { }
\begin{itemize}
    \item[$(i)$] $\mathscr{I}^{\mathrm{cf}}_\lambda$
         is a simple semigroup.

    \item[$(ii)$] An element $\alpha$ of the semigroup
         $\mathscr{I}^{\mathrm{cf}}_\lambda$
         is an idempotent if and only if $(x)\alpha=x$ for every
         $x\in\operatorname{dom}\alpha$.

    \item[$(iii)$] If $\varepsilon,\iota\in
          E(\mathscr{I}^{\mathrm{cf}}_\lambda)$,
          then $\varepsilon\leqslant\iota$ if and only if
          $\operatorname{dom}\varepsilon\subseteq
          \operatorname{dom}\iota$.

    \item[$(iv)$] The semilattice
          $E(\mathscr{I}^{\mathrm{cf}}_\lambda)$ is isomorphic to
          $(\mathscr{P}_{<\omega}(\lambda),\cup)$ under
          the mapping $(\varepsilon)h=\lambda\setminus
          \operatorname{dom}\varepsilon$.

    \item[$(v)$] Every maximal chain in
          $E(\mathscr{I}^{\mathrm{cf}}_\lambda)$ is an $\omega$-chain.

    \item[$(vi)$] $\alpha\mathscr{R}\beta$ in
         $\mathscr{I}^{\mathrm{cf}}_\lambda$ if and only if
         $\operatorname{dom}\alpha=\operatorname{dom}\beta$.

    \item[$(vii)$] $\alpha\mathscr{L}\beta$ in
         $\mathscr{I}^{\mathrm{cf}}_\lambda$ if and only if
         $\operatorname{ran}\alpha=\operatorname{ran}\beta$.

    \item[$(viii)$] $\alpha\mathscr{H}\beta$ in
         $\mathscr{I}^{\mathrm{cf}}_\lambda$ if and only if
         $\operatorname{dom}\alpha=\operatorname{dom}\beta$ and
         $\operatorname{ran}\alpha=\operatorname{ran}\beta$.

    \item[$(ix)$] $\alpha\mathscr{D}\beta$ for all
         $\alpha,\beta\in\mathscr{I}^{\mathrm{cf}}_\lambda$ and
         hence the semigroup $\mathscr{I}^{\mathrm{cf}}_\lambda$
         is bisimple.
\end{itemize}
\end{proposition}

\begin{proof}
$(i)$ We shall show that
$\mathscr{I}^{\mathrm{cf}}_\lambda\cdot\alpha
\cdot\mathscr{I}^{\mathrm{cf}}_\lambda=
\mathscr{I}^{\mathrm{cf}}_\lambda$ for every element $\alpha\in\mathscr{I}^{\mathrm{cf}}_\lambda$. Let $\alpha$ and $\beta$ are
arbitrary elements of the semigroup
$\mathscr{I}^{\mathrm{cf}}_\lambda$. We shall choose
elements
$\gamma,\delta\in\mathscr{I}^{\mathrm{cf}}_\lambda$ such that
$\gamma\cdot\alpha\cdot\delta=\beta$. We put
$\operatorname{dom}\gamma=\operatorname{dom}\beta$,
$\operatorname{ran}\gamma=\operatorname{dom}\alpha$,
$\operatorname{dom}\delta=\operatorname{ran}\alpha$ and
$\operatorname{ran}\delta=\operatorname{ran}\beta$. Since the sets
$\lambda\setminus\operatorname{dom}\alpha$ and
$\lambda\setminus\operatorname{dom}\beta$ are finite we conclude
that there exists a bijective map
$f\colon\operatorname{dom}\alpha\rightarrow\operatorname{dom}\beta$.
We put $\gamma=f$ and
$\left(\left((x)\gamma\right)\alpha\right)\delta=(x)\beta$ for all
$x\in\operatorname{dom}\beta$. Then we have that
$\gamma\cdot\alpha\cdot\delta=\beta$.

Statements $(ii)-(v)$ are trivial and they follow from the
definition of the semigroup $\mathscr{I}^{\mathrm{cf}}_\lambda$.
The proofs of $(vi)-(viii)$ follow trivially from the fact that
$\mathscr{I}^{\mathrm{cf}}_\lambda$ is a regular semigroup, and
Proposition 2.4.2 and Exercise 5.11.2 in \cite{Howie1995}.

$(ix)$ Let $\alpha$ and $\beta$ be arbitrary elements of the
semigroup $\mathscr{I}^{\mathrm{cf}}_\lambda$. Since the sets
$\lambda\setminus\operatorname{dom}\alpha$ and
$\lambda\setminus\operatorname{ran}\beta$ are finite we conclude
that there exists a bijective map
$\gamma\colon\operatorname{dom}\alpha\rightarrow
\operatorname{ran}\beta$. Then
$\gamma\in\mathscr{I}^{\mathrm{cf}}_\lambda$ and by statements
$(vi)$ and $(vii)$ we have that $\alpha\mathscr{R}\gamma$ and
$\beta\mathscr{L}\gamma$ in $\mathscr{I}^{\mathrm{cf}}_\lambda$ and
hence $\alpha\mathscr{D}\beta$ in
$\mathscr{I}^{\mathrm{cf}}_\lambda$.
\end{proof}

We denote the group of all bijective transformations of a set of cardinality $\lambda$ by $\mathscr{S}_\lambda$. Then we get the following:

\begin{corollary}\label{corollary-2.1a}
The group of units $H(\mathbb{I})$ of the semigroup
$\mathscr{I}^{\mathrm{cf}}_\lambda$ is isomorphic to
$\mathscr{S}_\lambda$.
\end{corollary}

For any idempotents $\varepsilon$ and $\iota$ of the semigroup
$\mathscr{I}^{\mathrm{cf}}_\lambda$ we denote:
\begin{equation*}
    H(\varepsilon,\iota)=\left\{\chi\in\mathscr{I}^{\mathrm{cf}}_\lambda
    \mid
    \chi\cdot\chi^{-1}=\varepsilon \hbox{~and~}
    \chi^{-1}\cdot\chi=\iota\right\} \qquad \hbox{~and~} \qquad
    H(\varepsilon)=H(\varepsilon,\varepsilon).
\end{equation*}
Proposition~\ref{proposition-2.1}$(viii)$ implies that the set
$H(\varepsilon,\iota)$ is a $\mathscr{H}$-class and the set
$H(\varepsilon)$ is a a maximal subgroup in
$\mathscr{I}^{\mathrm{cf}}_\lambda$ for all idempotents
$\varepsilon,\iota\in\mathscr{I}^{\mathrm{cf}}_\lambda$.

Corollary~\ref{corollary-2.1a} and Proposition~2.20 of \cite{CP}
imply the following:

\begin{corollary}\label{corollary-2.2}
Every maximal subgroup of the semigroup $\mathscr{I}^{\mathrm{cf}}_\lambda$ is isomorphic to~$\mathscr{S}_\lambda$.
\end{corollary}

\begin{proposition}\label{proposition-2.3}
$\left|\mathscr{I}^{\mathrm{cf}}_\lambda\right|=2^{|\lambda|}$.
\end{proposition}

\begin{proof}
Since $|\lambda\times\lambda|=|\lambda|$ we have that
$|\mathscr{S}_\lambda|\leqslant 2^{|\lambda\times\lambda|}=
2^{|\lambda|}$. Since $|\lambda\sqcup\lambda|=|\lambda|$ there
exists an injective map $f\colon\mathscr{P}(\lambda)\rightarrow
\mathscr{S}_{\lambda\sqcup\lambda}$ from the set
$\mathscr{P}(\lambda)$ of all subset of the cardinal $\lambda$ into
the group $\mathscr{S}_{\lambda\sqcup\lambda}$ defined in the following
way: $f(A)$ is a bijection on $\lambda\sqcup\lambda$ with support
$A\sqcup A$. Then we have that $|\mathscr{S}_\lambda|\geqslant
2^{|\lambda\sqcup\lambda|}=2^{|\lambda|}$ and hence
$|\mathscr{S}_\lambda|=2^{|\lambda|}$.

Since $\left|\mathscr{P}_{<\omega}(\lambda)\right|=
\left|\mathscr{P}_{<\omega}(\lambda)\times
\mathscr{P}_{<\omega}(\lambda)\right|=\lambda$ we conclude that
Theorem~2.20 from \cite{CP} and
Proposition~\ref{proposition-2.1}$(viii)$ imply that
\begin{equation*}
    \left|\mathscr{I}^{\mathrm{cf}}_\lambda\right|=
    \left|\mathscr{P}_{<\omega}(\lambda)\times
    \mathscr{P}_{<\omega}(\lambda)
    \times\mathscr{S}_\lambda\right|=
    \left|\mathscr{P}_{<\omega}(\lambda)\times
    \mathscr{P}_{<\omega}(\lambda)\right|
    \cdot|\mathscr{S}_\lambda|=
    |\lambda|\cdot 2^{|\lambda|}=
    2^{|\lambda|}.
\end{equation*}
\end{proof}

\begin{proposition}\label{proposition-2.4}
For every $\alpha,\beta\in\mathscr{I}^{\mathrm{cf}}_\lambda$, both
sets
\begin{equation*}
\left\{\chi\in\mathscr{I}^{\mathrm{cf}}_\lambda\mid
\alpha\cdot\chi=\beta\right\}\qquad \mbox{ and } \qquad
\left\{\chi\in\mathscr{I}^{\mathrm{cf}}_\lambda\mid
\chi\cdot\alpha=\beta\right\}
\end{equation*}
are finite. Consequently, every right translation and every left
translation by an element of the semigroup
$\mathscr{I}^{\mathrm{cf}}_\lambda$ is a finite-to-one map.
\end{proposition}

\begin{proof}
We denote
\begin{equation*}
A=\left\{\chi\in\mathscr{I}^{\mathrm{cf}}_\lambda\mid
\alpha\cdot\chi=\beta\right\} \quad \hbox{and} \quad
B=\left\{\chi\in\mathscr{I}^{\mathrm{cf}}_\lambda\mid
\alpha^{-1}\cdot\alpha\cdot\chi=\alpha^{-1}\cdot\beta\right\}.
\end{equation*}
Then $A\subseteq B$ and the restriction of any partial map $\chi\in B$
onto $\operatorname{dom}(\alpha^{-1}\cdot\alpha)$ coincides with the
partial map $\alpha^{-1}\cdot\beta$. Since every partial map from
$\mathscr{I}^{\mathrm{cf}}_\lambda$ has cofinite range and cofinite
domain we conclude that the set $B$ is finite and hence so is $A$.
\end{proof}

\begin{proposition}\label{proposition-2.5}
Each maximal chain $L$ of idempotents in $\mathscr{I}^{\mathrm{cf}}_\lambda$ coincides with the idempotent band $E(S)$ of a bicyclic subsemigroup
$S$ of $\mathscr{I}^{\mathrm{cf}}_\lambda$.
\end{proposition}

\begin{proof}
By Proposition~\ref{proposition-2.1}$(iii)$, the chain $L$ can be written as $L=\{\varepsilon_n\}_{n=1}^\infty$ where
$\varepsilon_1>\varepsilon_2>\cdots>\varepsilon_n>\cdots$. Since
every infinite subchain of an $\omega$-chain is also an $\omega$-chain
 we have that Proposition~\ref{proposition-2.1}$(v)$ implies that
$L$ is an $\omega$-chain. Then by
Proposition~\ref{proposition-2.1}$(iii)$ we get that
$\operatorname{dom}\varepsilon_i\setminus
\operatorname{dom}\varepsilon_{i+1}\neq\varnothing$ for all positive
integers $i$. Also, the maximality of $L$ implies that the set
$\operatorname{dom}\varepsilon_i\setminus
\operatorname{dom}\varepsilon_{i+1}$ is a singleton for all positive
integers $i$. For every positive integer $i$ we put
$\{x_i\}=\operatorname{dom}\varepsilon_i\setminus
\operatorname{dom}\varepsilon_{i+1}$. Then we put
$D=\operatorname{dom}\varepsilon_1\setminus
\bigcup_{i\in{\mathbb{N}}}\{x_i\}$ and define the partial maps
$\alpha\colon\lambda\rightharpoonup\lambda$ and
$\beta\colon\lambda\rightharpoonup\lambda$ as follows:
\begin{equation*}
    (x)\alpha=
\left\{
  \begin{array}{ll}
    x_{n+1}, & \hbox{ if } \, x=x_n\in\operatorname{dom}
    \varepsilon_1\setminus D \, \hbox{ and } \,n\geqslant 1;\\
    x, & \hbox{ if } \, x\in D;
  \end{array}
\right.
\end{equation*}
and
\begin{equation*}
    (x)\beta=
\left\{
  \begin{array}{ll}
    x_{n-1}, & \hbox{ if } \, x=x_n\in\operatorname{dom}
    \varepsilon_1\setminus D \, \hbox{ and } \, n>1;\\
    x, & \hbox{ if } \, x\in D.
  \end{array}
\right.
\end{equation*}
Since the set $\lambda\setminus\operatorname{dom}\varepsilon_1$ is
finite we have that
$\alpha,\beta\in\mathscr{I}^{\mathrm{cf}}_\lambda$ and
Remark~\ref{remark-1.1} implies the statement of our proposition.
\end{proof}

Proposition~\ref{proposition-2.5} and the Axiom of Choice imply the following
proposition.

\begin{proposition}\label{proposition-2.5a}
Each chain of idempotents in $\mathscr{I}^{\mathrm{cf}}_\lambda$ is contained in a bicyclic subsemigroup of $\mathscr{I}^{\mathrm{cf}}_\lambda$.
\end{proposition}

\begin{proposition}\label{proposition-2.6}
Let $\mathfrak{C}$ be a congruence on the semigroup
$\mathscr{I}^{\mathrm{cf}}_\lambda$. If there exist two
non-$\mathscr{H}$-equivalent elements
$\alpha,\beta\in\mathscr{I}^{\mathrm{cf}}_\lambda$ such that
$\alpha\mathfrak{C}\beta$, then $\mathfrak{C}$ is a group congruence
on $\mathscr{I}^{\mathrm{cf}}_\lambda$.
\end{proposition}

\begin{proof}
First we suppose that $\alpha$ and $\beta$ are distinct idempotents
of the semigroup $\mathscr{I}^{\mathrm{cf}}_\lambda$. Without loss
of generality we can assume that $\alpha$ and $\beta$ are compatible
and $\alpha\leqslant\beta$. Otherwise, replace $\alpha$ by $\alpha\cdot\beta$. Then by Proposition~\ref{proposition-2.5a} there exists a maximal chain $L$
in $E(\mathscr{I}^{\mathrm{cf}}_\lambda)$ such that $L$ contains the
elements $\alpha$ and $\beta$, and hence $L$ contained in a bicyclic subsemigroup $S$ of $\mathscr{I}^{\mathrm{cf}}_\lambda$. Then Corollary~1.32 of~\cite{CP}
implies that $\varepsilon\mathfrak{C}\iota$ for all elements
$\varepsilon$ and $\iota$ of the chain $L$.

Let $\nu$ be an arbitrary idempotent of the semigroup
$\mathscr{I}^{\mathrm{cf}}_\lambda$. Obviously, if
$\varepsilon,\iota\in L$ such that $\varepsilon\leqslant\iota$ then
$\varepsilon\cdot\nu\leqslant\iota\cdot\nu$. Since  ${\uparrow}e$ is
a finite subset of the free semilattice with unity
$(\mathscr{P}_{<\omega}(\lambda),\subseteq)$ for any
$e\in(\mathscr{P}_{<\omega}(\lambda),\subseteq)$, we have that
Proposition~\ref{proposition-2.1}$(iv)$ implies that $\nu L$ is an
infinite chain in $E(\mathscr{I}^{\mathrm{cf}}_\lambda)$. Then we
have that $\varepsilon\mathfrak{C}\iota$ for all
$\varepsilon,\iota\in\nu L$. We put $L_\nu=\nu
L\cup\{\nu\}\cup\{\mathbb{I}\}$. Then $L_\nu$ is a chain in
$E(\mathscr{I}^{\mathrm{cf}}_\lambda)$. Therefore by
Proposition~\ref{proposition-2.5a} we get that there exists a
maximal chain $L_{\max}$ in $E(\mathscr{I}^{\mathrm{cf}}_\lambda)$
which contains the chain $L_\nu$ and $L_{\max}$ is a band of a bicyclic
subsemigroup $S$ in $\mathscr{I}^{\mathrm{cf}}_\lambda$. Now Corollary~1.32 of
\cite{CP} implies that $\varepsilon\mathfrak{C}\iota$ for all
elements $\varepsilon$ and $\iota$ of the chain $L_\nu$. Hence
$\nu\mathfrak{C}\mathbb{I}$ and $\alpha\mathfrak{C}\mathbb{I}$ imply
that $\nu\mathfrak{C}\alpha$. Therefore all idempotents of the
semigroup $\mathscr{I}^{\mathrm{cf}}_\lambda$ are
$\mathfrak{C}$-equivalent. Since the semigroup
$\mathscr{I}^{\mathrm{cf}}_\lambda$ is inverse we conclude that
quotient semigroup $\mathscr{I}^{\mathrm{cf}}_\lambda/\mathfrak{C}$
contains only one idempotent and by Lemma~II.1.10 from
\cite{Petrich1984} the semigroup
$\mathscr{I}^{\mathrm{cf}}_\lambda/\mathfrak{C}$ is a group.

Suppose that $\alpha$ and $\beta$ are distinct
non--$\mathscr{H}$-equivalent elements of the semigroup
$\mathscr{I}^{\mathrm{cf}}_\lambda$ such that
$\alpha\mathfrak{C}\beta$. Then Proposition~\ref{proposition-2.1}
implies that at least one of the following conditions holds:
\begin{equation*}
    \alpha\alpha^{-1}\neq\beta\beta^{-1} \qquad \mbox{ or }
    \qquad \alpha^{-1}\alpha\neq\beta^{-1}\beta.
\end{equation*}
By Lemma~III.1.1 from \cite{Petrich1984} we have that
$\alpha^{-1}\mathfrak{C}\beta^{-1}$. Then
$\alpha\alpha^{-1}\mathfrak{C}\alpha\beta^{-1}$ and
$\beta\beta^{-1}\mathfrak{C}\alpha\beta^{-1}$ and hence
$\alpha\alpha^{-1}\mathfrak{C}\beta\beta^{-1}$. Similarly we get
that $\alpha^{-1}\alpha\mathfrak{C} \beta^{-1}\beta$. Then the first
part of the proof implies that $\mathfrak{C}$ is a group congruence
on $\mathscr{I}^{\mathrm{cf}}_\lambda$.
\end{proof}

\begin{theorem}\label{theorem-2.7}
Every non-trivial congruence on the semigroup
$\mathscr{I}^{\mathrm{cf}}_\lambda$ is a group congruence.
\end{theorem}

\begin{proof}
Let $\mathfrak{C}$ be a non-trivial congruence on the semigroup
$\mathscr{I}^{\mathrm{cf}}_\lambda$. Let $\alpha$ and $\beta$ be
distinct $\mathfrak{C}$-equivalent elements of the semigroup
$\mathscr{I}^{\mathrm{cf}}_\lambda$. If the elements $\alpha$ and
$\beta$ are not $\mathscr{H}$-equivalent then
Proposition~\ref{proposition-2.6} implies the statement of the
theorem.

Suppose that $\alpha\mathscr{H}\beta$. Then Theorem~2.20 from
\cite{CP} implies that without loss of generality we can assume that
$\alpha$ and $\beta$ are elements of the group of units
$H(\mathbb{I})$ of the semigroup $\mathscr{I}^{\mathrm{cf}}_\lambda$
and hence $\mathbb{I}\mathfrak{C}(\beta\alpha^{-1})$. We denote
$\gamma=\beta\alpha^{-1}$. Since $\mathbb{I}\neq\gamma$ we conclude
that there exists $x_0\in\lambda$ such that $(x_0)\gamma\neq x_0$.
We define $\varepsilon$ to be an identity selfmap of the set
$\lambda\setminus\{x_0\}$. Then
$\varepsilon\in\mathscr{I}^{\mathrm{cf}}_\lambda$ and
$(\varepsilon\cdot\mathbb{I})\mathfrak{C}(\varepsilon\cdot\gamma)$.
Since $(x_0)\gamma\neq x_0$ we have that
Proposition~\ref{proposition-2.1}$(viii)$ implies that the elements
$\varepsilon$ and $\varepsilon\cdot\gamma$ are not
$\mathscr{H}$-equivalent. Then by Proposition~\ref{proposition-2.6}
we get that $\mathfrak{C}$ is a group congruence on
$\mathscr{I}^{\mathrm{cf}}_\lambda$.
\end{proof}


\section{On the least group congruence on the semigroup $\mathscr{I}^{\mathrm{cf}}_\lambda$}

Every inverse semigroup $S$ admits the least group congruence $\sigma$ (see \cite[Section~III]{Petrich1984}):
\begin{equation*}
    s\sigma t \quad \hbox{if and only if \quad there exists an idempotent} \quad e\in S \quad \hbox{such that} \quad se=te.
\end{equation*}

Theorem~\ref{theorem-2.7} implies that every non-injective homomorphism $h\colon\mathscr{I}^{\mathrm{cf}}_\lambda \rightarrow S$ from the semigroup $\mathscr{I}^{\mathrm{cf}}_\lambda$ into an arbitrary semigroup $S$ generates a group congruence $\mathfrak{h}$ on $\mathscr{I}^{\mathrm{cf}}_\lambda$. In this section we describe the structure of the quotient semigroup $\mathscr{I}^{\mathrm{cf}}_\lambda/\sigma$.

\begin{proposition}\label{proposition-3.1}
If $\alpha\sigma\beta$ in $\mathscr{I}^{\mathrm{cf}}_\lambda$ then
\begin{equation*}
    |\lambda\setminus\operatorname{dom}\alpha|- |\lambda\setminus\operatorname{ran}\alpha|= |\lambda\setminus\operatorname{dom}\beta|- |\lambda\setminus\operatorname{ran}\beta|.
\end{equation*}
\end{proposition}

\begin{proof}
Let $\varepsilon$ be an idempotent of the semigroup $\mathscr{I}^{\mathrm{cf}}_\lambda$ such that $\alpha\varepsilon=\beta\varepsilon$. We shall show that the statement of the proposition holds for the elements $\alpha$ and $\alpha\varepsilon$.

Without loss of generality we can assume that $\varepsilon\leqslant \alpha^{-1}\alpha$, i.e., $\operatorname{dom}\varepsilon\subseteq \operatorname{dom}(\alpha^{-1}\alpha)$. Since $\alpha$ is an injective partial map with $|\lambda\setminus\operatorname{dom}\alpha|<\infty$ and $|\lambda\setminus\operatorname{ran}\alpha|<\infty$, and $\varepsilon$ is an identity map of the cofinite subset $\operatorname{dom}\varepsilon$ in $\lambda$ we conclude that
\begin{equation*}
    |\lambda\setminus\operatorname{dom}\alpha|- |\lambda\setminus\operatorname{ran}\alpha|= |\lambda\setminus\operatorname{dom}(\alpha\varepsilon)|- |\lambda\setminus\operatorname{ran}(\alpha\varepsilon)|.
\end{equation*}
This implies the statement of the proposition.
\end{proof}

For an arbitrary element $\alpha$ of the semigroup $\mathscr{I}^{\mathrm{cf}}_\lambda$ we denote
\begin{equation*}
    \overline{d}(\alpha)=|\lambda\setminus\operatorname{dom}\alpha| \qquad \hbox{and} \qquad \overline{r}(\alpha)=|\lambda\setminus\operatorname{ran}\alpha|.
\end{equation*}

\begin{proposition}\label{proposition-3.2}
If $\alpha$ and $\beta$ are arbitrary elements of the semigroup $\mathscr{I}^{\mathrm{cf}}_\lambda$ then
\begin{equation*}
    \overline{d}(\alpha\beta)-\overline{r}(\alpha\beta)=
    \overline{d}(\alpha)-\overline{r}(\alpha)+
    \overline{d}(\beta)-\overline{r}(\beta).
\end{equation*}
\end{proposition}

\begin{proof}
We consider four cases.

$(1)$~First we consider the case when $\operatorname{ran}\alpha\subseteq \operatorname{dom}\beta$. We put $k=\overline{r}(\alpha)-\overline{d}(\beta)$. Then the definition of the semigroup $\mathscr{I}^{\mathrm{cf}}_\lambda$ implies that $k\geqslant 0$, $\overline{d}(\alpha\beta)=\overline{d}(\alpha)$, $\overline{r}(\alpha\beta)=\overline{r}(\beta)-k$, and hence in this case we get that
\begin{equation*}
    \overline{d}(\alpha\beta)-\overline{r}(\alpha\beta)=
    \overline{d}(\alpha)-\overline{r}(\alpha)+
    \overline{d}(\beta)-\overline{r}(\beta).
\end{equation*}

$(2)$~Suppose that the case when $\operatorname{dom}\beta\subseteq \operatorname{ran}\alpha$ holds. We put $k=\overline{d}(\beta)-\overline{r}(\alpha)$. Then the definition of the semigroup $\mathscr{I}^{\mathrm{cf}}_\lambda$ implies that $k\geqslant 0$, $\overline{d}(\alpha\beta)=\overline{d}(\alpha)+k$, $\overline{r}(\alpha\beta)=\overline{r}(\beta)$, and hence in this case we have that
\begin{equation*}
    \overline{d}(\alpha\beta)-\overline{r}(\alpha\beta)=
    \overline{d}(\alpha)-\overline{r}(\alpha)+
    \overline{d}(\beta)-\overline{r}(\beta).
\end{equation*}

$(3)$~Now we consider the case $(\lambda\setminus\operatorname{ran}\alpha)\cap (\lambda\setminus\operatorname{dom}\beta)\neq\varnothing$, $\operatorname{ran}\alpha\nsubseteq \operatorname{dom}\beta$ and
$\operatorname{dom}\beta\nsubseteq \operatorname{ran}\alpha$. Then the definition of the semigroup $\mathscr{I}^{\mathrm{cf}}_\lambda$ implies that there exist positive integers $i$, $j$ and $k$ such that $i=\left|(\lambda\setminus\operatorname{ran}\alpha)\setminus (\lambda\setminus\operatorname{dom}\beta)\right|$, $j=\left|(\lambda\setminus\operatorname{ran}\alpha)\cap (\lambda\setminus\operatorname{dom}\beta)\right|$ and $k=\left|(\lambda\setminus\operatorname{dom}\beta)\setminus (\lambda\setminus\operatorname{ran}\alpha)\right|$. Then we have that $\overline{r}(\alpha)=i+j$, $\overline{d}(\beta)=j+k$, $\overline{d}(\alpha\beta)=\overline{d}(\alpha)+k$ and $\overline{r}(\alpha\beta)=\overline{r}(\beta)+i$. Therefore, in this case  we get that
\begin{equation*}
    \overline{d}(\alpha\beta)-\overline{r}(\alpha\beta)=
    \overline{d}(\alpha)-\overline{r}(\alpha)+
    \overline{d}(\beta)-\overline{r}(\beta).
\end{equation*}

$(4)$~In the case when $(\lambda\setminus\operatorname{ran}\alpha)\cap (\lambda\setminus\operatorname{dom}\beta)=\varnothing$ we have that the definition of the semigroup $\mathscr{I}^{\mathrm{cf}}_\lambda$ implies that $\overline{d}(\alpha\beta)=\overline{d}(\alpha)+\overline{d}(\beta)$, $\overline{r}(\alpha\beta)=\overline{r}(\alpha)+\overline{r}(\beta)$, and hence we get that
\begin{equation*}
    \overline{d}(\alpha\beta)-\overline{r}(\alpha\beta)=
    \overline{d}(\alpha)-\overline{r}(\alpha)+
    \overline{d}(\beta)-\overline{r}(\beta).
\end{equation*}
This completes the proof of the proposition.
\end{proof}

On the semigroup $\mathscr{I}^{\mathrm{cf}}_\lambda$ we define a relation $\sim_{\mathfrak{d}}$ in the following way:
\begin{equation*}
    \alpha\sim_{\mathfrak{d}}\beta \qquad \hbox{if and only if} \qquad \overline{d}(\alpha)-\overline{r}(\alpha)=
    \overline{d}(\beta)-\overline{r}(\beta),
\end{equation*}
for $\alpha,\beta\in\mathscr{I}^{\mathrm{cf}}_\lambda$.

\begin{proposition}\label{proposition-3.3}
Let $\lambda$ be an infinite cardinal. Then $\sim_{\mathfrak{d}}$ is a congruence on the semigroup $\mathscr{I}^{\mathrm{cf}}_\lambda$ and moreover the quotient semigroup $\mathscr{I}^{\mathrm{cf}}_\lambda/\sim_{\mathfrak{d}}$ is isomorphic to the additive group of integers $\mathbb{Z}(+)$.
\end{proposition}

\begin{proof}
Simple verifications and Proposition~\ref{proposition-3.2} imply that $\sim_{\mathfrak{d}}$ is a congruence on the semigroup $\mathscr{I}^{\mathrm{cf}}_\lambda$. We define a homomorphism $h\colon \mathscr{I}^{\mathrm{cf}}_\lambda\rightarrow \mathbb{Z}(+)$ by the formula $(\alpha)h=\overline{d}(\alpha)-\overline{r}(\alpha)$. Then the definitions of the semigroup $\mathscr{I}^{\mathrm{cf}}_\lambda$ and the congruence $\sim_{\mathfrak{d}}$ on $\mathscr{I}^{\mathrm{cf}}_\lambda$, and Proposition~\ref{proposition-3.2} imply that thus defined map $h$ is a surjective homomorphism and moreover $(\alpha)h=(\beta)h$ if and only if $\alpha\sim_{\mathfrak{d}}\beta$ in $\mathscr{I}^{\mathrm{cf}}_\lambda$. This completes the proof of the proposition.
\end{proof}

\begin{proposition}\label{proposition-3.4}
Let $\lambda$ be an infinite cardinal. Then for every element $\beta$ of the semigroup $\mathscr{I}^{\mathrm{cf}}_\lambda$ such that $\overline{d}(\beta)= \overline{r}(\beta)$ there exists an element $\alpha$ of the group of units of $\mathscr{I}^{\mathrm{cf}}_\lambda$ such that $\alpha\sigma\beta$.
\end{proposition}

\begin{proof}
Fix an arbitrary element $\beta$ of the semigroup $\mathscr{I}^{\mathrm{cf}}_\lambda$. Without loss of generality we can assume that $\overline{d}(\beta)= \overline{r}(\beta)=k>0$. Let $\{x_1,\ldots,x_k\}= \lambda \setminus\operatorname{dom}\beta$ and $\{y_1,\ldots,y_k\}= \lambda \setminus\operatorname{ran}\beta$. We define a map $\alpha\colon\lambda\rightarrow \lambda$ in the following way:
\begin{equation*}
    (x)\alpha=
\left\{
  \begin{array}{ll}
    (x)\beta, & \hbox{if~~} ~x\in\operatorname{dom}\beta; \\
    y_i, & \hbox{if~~}      ~x=x_i, \; i=1,\ldots,k.
  \end{array}
\right.
\end{equation*}
Then $\alpha$ is an element of the group of units of the semigroup $\mathscr{I}^{\mathrm{cf}}_\lambda$ and it is obviously that $\alpha\varepsilon=\beta\varepsilon$, where $\varepsilon$ is the identity map of the set $\operatorname{ran}\beta$.
\end{proof}

For every $\alpha\in\mathscr{S}_\lambda$ we denote $\operatorname{supp}(\alpha)=\{ x\in\lambda\mid (x)\alpha\neq x\}$. We define
\begin{equation*}
    \mathscr{S}^\infty_\lambda=\left\{\alpha\in\mathscr{S}_\lambda\mid \operatorname{supp}(\alpha) \hbox{ is finite}\right\}.
\end{equation*}

We observe that the Schreier--Ulam theorem (see \cite[Theorem~11.3.4]{Scott1987}) implies that $\mathscr{S}^\infty_\lambda$ is a normal subgroup of $\mathscr{S}_\lambda$ and hence $\mathscr{S}_\lambda/\mathscr{S}^\infty_\lambda$ is a group.

Later on, when $\mathfrak{C}$ is a congruence on a semigroup $S$ we shall denote the \emph{natural homomorphism} generated by the congruence $\mathfrak{C}$ on $S$ by $\pi_{\mathfrak{C}}\colon S\rightarrow S/\mathfrak{C}$.

The definition of the least group congruence $\sigma$ on the semigroup $\mathscr{I}^{\mathrm{cf}}_\lambda$ implies the following proposition.

\begin{proposition}\label{proposition-3.5}
Let $\lambda$ be an infinite cardinal. Then the homomorphic image $(H(\mathbb{I}))\pi_{\sigma}$ of the group of units $H(\mathbb{I})$ of $\mathscr{I}^{\mathrm{cf}}_\lambda$ under the natural homomorphism $\pi_\sigma\colon \mathscr{I}^{\mathrm{cf}}_\lambda\rightarrow\mathscr{I}^{\mathrm{cf}}_\lambda/\sigma$ is isomorphic to the quotient group $\mathscr{S}_\lambda/\mathscr{S}^\infty_\lambda$.
\end{proposition}

\begin{theorem}\label{theorem-3.6}
Let $\lambda$ be an infinite cardinal. Then the following conditions hold:
\begin{itemize}
  \item[$(i)$] $(H(\mathbb{I}))\pi_{\sigma}= \mathscr{S}_\lambda/\mathscr{S}^\infty_\lambda$ is a normal subgroup of the group $\mathscr{I}^{\mathrm{cf}}_\lambda/\sigma$;
  \item[$(ii)$] The group $\mathscr{I}^{\mathrm{cf}}_\lambda/\sigma$ contains the infinite cyclic subgroup $G$ (i.e., the additive group of integers $\mathbb{Z}(+)$) such that $G\cap \mathscr{S}_\lambda/\mathscr{S}^\infty_\lambda=\{e\}$, where $e$ is the unit of the group $\mathscr{I}^{\mathrm{cf}}_\lambda/\sigma$;
  \item[$(iii)$] $\mathscr{S}_\lambda/\mathscr{S}^\infty_\lambda\cdot G= \mathscr{I}^{\mathrm{cf}}_\lambda/\sigma$.
\end{itemize}
and hence the group $\mathscr{I}^{\mathrm{cf}}_\lambda/\sigma$ is isomorphic to the semidirect product $\mathscr{S}_\lambda/\mathscr{S}^\infty_\lambda\ltimes \mathbb{Z}(+)$.
\end{theorem}

\begin{proof}
$(i)$ Since $\sigma\subseteq\sim_{\mathfrak{d}}$ we conclude that Theorem~1.6 of \cite{CP} implies that there exists a unique homomorphism $g\colon \mathscr{I}^{\mathrm{cf}}_\lambda/\sigma\rightarrow G$ such that the following diagram
\begin{equation*}
    \xymatrix{ \mathscr{I}^{\mathrm{cf}}_\lambda \ar[r]^{\pi_{\sigma}}\ar[dr]_{\pi_{\sim_{\mathfrak{d}}}} & \mathscr{I}^{\mathrm{cf}}_\lambda/\sigma\ar[d]^g\\ & G }
\end{equation*}
commutes. Then by Proposition~\ref{proposition-3.5} we have that the homomorphic image $(H(\mathbb{I}))\pi_{\sigma}$ of the group of units $H(\mathbb{I})$ of $\mathscr{I}^{\mathrm{cf}}_\lambda$ under the natural homomorphism $\pi_\sigma\colon \mathscr{I}^{\mathrm{cf}}_\lambda\rightarrow\mathscr{I}^{\mathrm{cf}}_\lambda/\sigma$ is isomorphic the the quotient group $\mathscr{S}_\lambda/\mathscr{S}^\infty_\lambda$. Now Propositions~\ref{proposition-3.4} and ~\ref{proposition-3.5} imply that the subgroup $(H(\mathbb{I}))\pi_{\sigma}= \mathscr{S}_\lambda/\mathscr{S}^\infty_\lambda$ of the group $\mathscr{I}^{\mathrm{cf}}_\lambda/\sigma$ is the kernel of the homomorphism $g\colon\mathscr{I}^{\mathrm{cf}}_\lambda/\sigma\rightarrow G$, and hence $(H(\mathbb{I}))\pi_{\sigma}= \mathscr{S}_\lambda/\mathscr{S}^\infty_\lambda$ is a normal subgroup of $\mathscr{I}^{\mathrm{cf}}_\lambda/\sigma$.

$(ii)$ Fix an arbitrary $\alpha\in\mathscr{I}^{\mathrm{cf}}_\lambda$ such that $|\lambda\setminus\operatorname{dom}\alpha|=1$ and $\operatorname{ran}\alpha=\lambda$. Then the definition of the congruence $\sim_{\mathfrak{d}}$ on $\mathscr{I}^{\mathrm{cf}}_\lambda$ implies that the element $\alpha^n$ is not $\sim_{\mathfrak{d}}$-equivalent to any element of the group of units $H(\mathbb{I})$ for every non-zero integer $n$, and hence by Proposition~\ref{proposition-3.5} we get that $((\alpha)\pi_{\sigma})^n\notin \mathscr{S}_\lambda/\mathscr{S}^\infty_\lambda$. This implies that $\{((\alpha)\pi_{\sim_{\mathfrak{d}}})^n\mid n\in\mathbb{Z}\}\cap\mathscr{S}_\lambda/\mathscr{S}^\infty_\lambda=\{e\}$, where $e$ is the unit of the group $\mathscr{I}^{\mathrm{cf}}_\lambda/\sigma$.
Also, it is obvious that $(\alpha^n)\pi_{\sim_{\mathfrak{d}}}=n\in G$ and $\{(\alpha^n)\pi_{\sim_{\mathfrak{d}}}\mid n\in\mathbb{Z}\}$ is a cyclic subgroup of $\mathscr{S}_\lambda/\mathscr{S}^\infty_\lambda$.

$(iii)$ Fix an arbitrary element $x$ in $\mathscr{I}^{\mathrm{cf}}_\lambda/\sigma$. Let $\xi$ be an arbitrary element of the semigroup $\mathscr{I}^{\mathrm{cf}}_\lambda$ be such that $(\xi)\pi_{\sigma}=x$. If $\overline{d}(\xi)=\overline{r}(\xi)$ then by Proposition~\ref{proposition-3.4} we have that $\xi\sigma\beta$ for some element $\beta$ from the group of units of $\mathscr{I}^{\mathrm{cf}}_\lambda$, and hence we get that $x=(\beta)\pi_\sigma\cdot e\in \mathscr{S}_\lambda/\mathscr{S}^\infty_\lambda\cdot G$, where $e$ is the unit of the group $\mathscr{I}^{\mathrm{cf}}_\lambda/\sigma$. Suppose that $\overline{d}(\xi)-\overline{r}(\xi)=n\neq 0$. Then by Proposition~\ref{proposition-3.2} we have that $\overline{d}\left(\xi\cdot(\alpha^{-1})^n\right)- \overline{r}\left(\xi\cdot(\alpha^{-1})^n\right)=0$. Now, Proposition~\ref{proposition-3.4} implies that the element  $\xi\cdot(\alpha^{-1})^n$ is $\sigma$-equivalent to some element $\beta$ of the group of units $H(\mathbb{I})$ of $\mathscr{I}^{\mathrm{cf}}_\lambda$. Then we have that $(\xi\cdot(\alpha^{-1})^n)\pi_\sigma=(\beta)\pi_\sigma$ and since $\mathscr{I}^{\mathrm{cf}}_\lambda/\sigma$ is a group we get that $x=(\xi)\pi_\sigma=(\beta)\pi_\sigma\cdot (\alpha^n)\pi_\sigma\in \mathscr{S}_\lambda/\mathscr{S}^\infty_\lambda\cdot G$. This implies that $\mathscr{S}_\lambda/\mathscr{S}^\infty_\lambda\cdot G= \mathscr{I}^{\mathrm{cf}}_\lambda/\sigma$.

The last statement of the theorem follows from statements $(i)$--$(iii)$ and Exercise~2.5.3 from~\cite{DixonMortimer}.
\end{proof}

\begin{remark}\label{remark-3.7}
Proposition~\ref{proposition-3.3} implies that for every infinite cardinal $\lambda$ the group $\mathscr{I}^{\mathrm{cf}}_\lambda/\sigma$ has infinitely many normal subgroups and hence the semigroup $\mathscr{I}^{\mathrm{cf}}_\lambda$ has infinitely many group congruences.
\end{remark}

\begin{acknowledgement}
The authors are grateful to the referee for several useful comments
and suggestions.
\end{acknowledgement}


\end{document}